\documentclass[letter,dvipsnames,12pt]{article}

\usepackage[english]{babel}
\usepackage[utf8]{inputenc}
\usepackage[normalem]{ulem}

\usepackage[margin=1 in, top=1 in, bottom= 1 in]{geometry}

\makeatletter
\g@addto@macro\normalsize{%
  \setlength\abovedisplayskip{7pt}
  \setlength\belowdisplayskip{7pt}
  \setlength\abovedisplayshortskip{7pt}
  \setlength\belowdisplayshortskip{7pt}
}
\makeatother

\usepackage{tocloft}

\usepackage{amsfonts}
\usepackage{mathrsfs}
\usepackage{bbm}
\usepackage{latexsym}
\usepackage{amssymb}
\usepackage{mathtools}

\usepackage{enumitem}
\setlist{nolistsep} 	
\usepackage{amsthm}

\usepackage{xcolor}
\definecolor{Color1}{rgb}{0.0, 0.42, 0.47}
\definecolor{Color2}{rgb}{0.78, 0.11, 0.0}

\usepackage{titlesec}
\titleformat{\section}
  {\large\center\bfseries}
  {\thesection.}{.7em}{}
\titlespacing*{\section}{0pt}{3.5ex plus 0ex minus 0ex}{1.5ex plus 0ex}
\titleformat{\subsection}
  {\center\bfseries}
  {\thesubsection.}{.7em}{}
\titlespacing*{\subsection}{0pt}{3.5ex plus 0ex minus 0ex}{1.5ex plus 0ex}
\titleformat{\subsubsection}
  {\center\bfseries}
  {\thesubsubsection.}{.7em}{}
\titlespacing*{\subsubsection}{0pt}{3.5ex plus 0ex minus 0ex}{1.5ex plus 0ex}

\addto\captionsenglish{}

\usepackage{titling}
\makeatletter
\renewenvironment{abstract}{
\begin{center}
{\bfseries \large\abstractname\vspace{\z@}}
\end{center}
\quotation
}
\makeatother

\usepackage[linktocpage=true]{hyperref}
\usepackage[capitalize]{cleveref}
\hypersetup{citecolor = Color1,colorlinks,
			linkcolor = black,
			urlcolor = Color2}

\newtheoremstyle{plain}{3mm}{3mm}{\slshape}{}{\bfseries}{.}{.5em}{}
\newtheoremstyle{definition}{2mm}{2mm}{}{}{\bfseries}{.}{.5em}{}
\theoremstyle{plain}
	
\newtheorem{Theorem}{Theorem}

\newtheorem{claim}{Claim}
\theoremstyle{definition}

\newtheorem{Remark}[Theorem]{Remark}

\usepackage{chngcntr}
\counterwithin{Theorem}{section}

\numberwithin{equation}{section}

\allowdisplaybreaks

\newcommand{\Erdos}{Erd\H{o}s}
\newcommand{\Folner}{F\o{}lner}
\newcommand{\Szemeredi}{Szemer\'{e}di}

\newcommand{\N}{\mathbb{N}}

\newcommand{\C}{\mathbb{C}}

\renewcommand{\epsilon}{\varepsilon}
\renewcommand{\leq}{\leqslant}

\renewcommand{\setminus}{\backslash}

\newcommand{\1}{\mathbf{1}}
\renewcommand{\d}{\,\mathsf{d}}

\author{By~~{\scshape Bryna Kra}~~and~~{\scshape Joel~Moreira}~~and~~{\scshape Florian~K.~Richter}\\~~and~~{\scshape Donald Robertson}}

\title{\textbf{A short proof of \Erdos{}'s $B+C$ conjecture}
{\normalfont \normalsize \em{ Dedicated to Vitaly Bergelson on his 75$^{\mathrm{th}}$ birthday}}
}

\begin{document}

\maketitle

\abstract{We give a short proof of the fact that every set of natural numbers with positive upper Banach density contains the sum of two infinite sets.
The approach simplifies earlier existing proofs.}


\section{Introduction}
\label{sec_intro}

In the late 1970s and early 1980s, 
\Erdos{} sought to characterize the types of subsets of integers that contain infinite sumsets, a line of inquiry that culminated in two noteworthy conjectures. Both have since been resolved and we state them as theorems below.

A \emph{\Folner{} sequence} is a sequence $\Phi=(\Phi_N)_{N\in\N}$ of finite subsets of $\N$ satisfying 
\[
\lim_{N \to \infty} \frac{|\Phi_N \cap (\Phi_N + 1)|}{|\Phi_N|} = 1.
\]
A set $A \subset \N$ has \emph{positive upper Banach density} if
\[
\lim_{N \to \infty} \dfrac{|A \cap \Phi_N|}{|\Phi_N|} > 0
\]
for some \Folner{} sequence $\Phi=(\Phi_N)_{N\in\N}$.

We begin with the weaker conjecture, which was the second conjecture in chronological order.
\begin{Theorem}[\Erdos{}'s $2^{\mathrm{nd}}$ sumset conjecture]
\label{thm_mrr}
For any $A\subset\N$ with positive upper Banach density, there exist infinite sets $B,C\subset \N$ such that $B+C=\{b+c:b\in B,~c\in C\}\subset A$.
\end{Theorem}

Theorem~\ref{thm_mrr} was first proved in~\cite{MRR19}. The argument was later simplified significantly by Host in~\cite{Host19}. Another proof, following a different and partly easier argument, was given in Section~3 of~\cite{KMRR24a}.

\begin{Theorem}[\Erdos{}'s $1^{\mathrm{st}}$ sumset conjecture]
\label{thm_kmrr}
For any $A\subset\N$ with positive upper Banach density, there exist $b_1<b_2<b_3<\ldots\in\N$ and $t\in\N$ such that $\{b_i+b_j: i,j\in\N,~i\neq j\}\subset A-t$.
\end{Theorem}

Note that Theorem~\ref{thm_kmrr} contains Theorem~\ref{thm_mrr} as a special case. Theorem~\ref{thm_kmrr} was proved in~\cite{KMRR24b}.

The purpose of this note is to provide a new and relatively short proof of Theorem~\ref{thm_mrr}.
The ideas behind this proof are based on recent advances made on Theorem~\ref{thm_kmrr} and its generalizations  in~\cite{KMRR25,KMRR26}.

\paragraph{Acknowledgements}
BK was partially supported by NSF grant DMS-2348315, JM was supported by EPSRC Frontier Research Guarantee grant EP/Y014030/1, and FKR was supported by the Swiss National Science Foundation grant TMSGI2-211214.

\section{Preliminaries from ergodic theory}
\label{sec_prelims}

Throughout, by a \emph{measure-preserving system} we mean a tuple $(X,\mathscr{B}_X,\mu,T)$ where $X$ is a compact metric space, $\mathscr{B}_X$ is the Borel $\sigma$-algebra on $X$, $T\colon X\to X$ is a homeomorphism, and $\mu\colon \mathscr{B}_X\to[0,1]$ is a Borel probability measure on $X$ invariant under the transformation $T$.

Given a \Folner{} sequence $\Phi=(\Phi_N)_{N\in\N}$, a point $x\in X$ is  \emph{generic for $\mu$ along $\Phi$} if
\begin{equation}\label{eq_generic}
\lim_{N\to\infty} \frac{1}{|\Phi_N|}\sum_{n\in\Phi_N} f(T^n x)=\int_X f\d\mu \qquad\text{ for all } f\in C(X),
\end{equation}
where $C(X)$ denotes the space of all (complex-valued) continuous functions on $X$.
\begin{Remark}
\label{rem_generic_points_along_Folner}
When the measure $\mu$ is ergodic, there exists a subsequence $\Psi$ of $\Phi$ along which almost every point is generic. Indeed, the mean ergodic theorem implies that \eqref{eq_generic} holds in $L^2$ norm for every F\o lner sequence and every $f\in C(X)$. 
As norm convergence implies almost sure convergence along a subsequence, and the space $C(X)$ has a countable dense set, 
a diagonal argument produces the desired subsequence $\Psi$.
\end{Remark}

A nonzero function $f\in L^2(X,\mathscr{B}_X,\mu)$ is a \emph{measurable eigenfunction} with \emph{eigenvalue} $\theta\in\C$ if $f(Tx)= \theta f(x)$ holds for $\mu$-almost every $x\in X$. A nonzero function $f\in C(X)$ is a \emph{topological eigenfunction} if $f(Tx)=\theta f(x)$ for every $x\in X$.
We emphasize that every topological eigenfunction is a measurable one, but not necessarily the other way around.

It is well known that the closed subspace of $L^2(X,\mathscr{B}_X,\mu)$ spanned by measurable eigenfunctions is of the form $L^2(X,\mathscr{K},\mu)$ for a sub-$\sigma$-algebra $\mathscr{K}$ of $\mathscr{B}_X$ (see~\cite[Theorem~6.10 and Theorem~C.11]{EW11}). 
Write $\mathbb{E}(f\mid\mathscr{K})$ for the orthogonal projection in $L^2(X,\mathscr{B}_X,\mu)$ of $f$ on $L^2(X,\mathscr{K},\mu)$.
If $f \ge 0$ then $\mathbb{E}(f\mid\mathscr{K}) \ge 0$.

\begin{Remark}
\label{rem_support_of_conditional_expectation}
Given a measurable set $E\subset X$, let $D=\{x\in X: \mathbb{E}(\1_E\mid\mathscr{K})(x)>0\}$ denote the support of $\mathbb{E}(\1_E\mid\mathscr{K})$. 
Since $X\setminus D\in\mathscr{K}$ and $\mathbb{E}(\1_E\mid\mathscr{K}) \ge 0$, it follows that
\[
0=\int_{X\setminus D}\mathbb{E}(\1_E\mid\mathscr{K})\d\mu=\mu(E\setminus D).
\]
We conclude that for $\mu$-almost every $x\in E$ we have $\mathbb{E}(\1_E\mid\mathscr{K})(x)>0$. We use this property in the proof of our main result.
\end{Remark}

We call $f\in L^2(X,\mathscr{B}_X,\mu)$ a \emph{weak mixing} function if for all $g\in L^2(X,\mathscr{B}_X,\mu)$ and all \Folner{} sequences $\Phi=(\Phi_N)_{N\in\N}$ we have
\begin{equation}\label{eq_wmdef}
    \lim_{N\to\infty} \frac{1}{|\Phi_N|}\sum_{n\in\Phi_N} |\langle T^nf ,g\rangle|=0.
\end{equation}

\begin{Remark}\label{remark_wm}
    Note that the limit in \eqref{eq_wmdef} is clearly zero whenever $g$ is orthogonal to $\{ T^n f : n \in \N \}$. This implies that a function is weak mixing if and only if
\[
\lim_{N\to\infty} \frac{1}{|\Phi_N|}\sum_{n\in\Phi_N} |\langle T^nf ,f\rangle|=0.
\]
\end{Remark}

The following fundamental result is frequently used and is implicit in the work of Koopman and von  Neumann~\cite{Koopman_vonNeumann32} 
(cf.\ \cite[Theorem~2.3.4]{Krengel85} or~\cite[Chapter 4,~Proposition~19]{HK18}).

\begin{Theorem}
\label{thm_JdLG_decomposition}
For any $f\in L^2(X,\mathscr{B}_X,\mu)$ we have
\[
f~\text{is weak mixing}~~\iff~~ \mathbb{E}(f\mid\mathscr{K})=0.
\]
\end{Theorem}

A crucial tool for translating a combinatorial problem into ergodic theory is Furstenberg's correspondence principle, introduced in his proof~\cite{Furstenberg77} of \Szemeredi{}'s theorem (see also~\cite{Furstenberg81a}). We utilize the following enhanced variant, which allows us to assume that all eigenfunctions are topological eigenfunctions.

\begin{Theorem}[Correspondence principle]
\label{thm_correspondence_principle}
For any set $A\subset \N$ with positive upper Banach density there exist a \Folner{} sequence $\Phi=(\Phi_N)_{N\in\N}$, an ergodic measure-preserving system $(X,\mu,T)$, a clopen set $E\subset X$, and a point $x\in X$ such that:
\begin{enumerate}	
\item
The point $x$ is generic for $\mu$ along $\Phi$.
\item
The measure $\mu(E)>0$.
\item
$A=\{n\in\N: T^n x\in E\}$.
\item Every measurable eigenfunction is equal $\mu$-almost everywhere to a topological eigenfunction.
\end{enumerate}
\end{Theorem}

Theorem~\ref{thm_correspondence_principle} is implicit in~\cite{KMRR24a} as a combination of~\cite[Theorem 2.10]{KMRR24a} and~\cite[Lemma 5.8]{KMRR24a}. The former establishes the first three properties as in Furstenberg's original correspondence principle.
The latter is then used to produce a topological  model of that system that additionally satisfies the fourth property.

\section{Reformulation from combinatorics to ergodic theory}

Using Furstenberg's correspondence principle, we reduce the existence of sumsets to the following dynamical result, versions of which have appeared in all of~\cite{KMRR24a, KMRR24b, KMRR25, KMRR26}.

\begin{Theorem}[Main dynamical result]
\label{thm_main_dynamical}
Let $(X,\mu,T)$ be an ergodic measure-preserving system and let $x\in X$ be a point that is generic for $\mu$ along a \Folner{} sequence $\Phi=(\Phi_N)_{N\in\N}$. 
Assume every measurable eigenfunction of $(X,\mu,T)$ is equal $\mu$-almost everywhere to a topological eigenfunction. Then for any clopen set $E\subset X$ with $\mu(E)>0$ there exist sequences $b_1<b_2<\ldots\in\N$ and $c_1<c_2<\ldots\in\N$ such that
\begin{equation}
\label{eqn_fuzzy_Erdos_square}
\lim_{j\to\infty}\lim_{i\to\infty} T^{b_i+c_j}x \in E\qquad\text{and}\qquad \lim_{i\to\infty}\lim_{j\to\infty} T^{b_i+c_j}x \in E.
\end{equation}
\end{Theorem}

\begin{proof}[Proof that Theorem~\ref{thm_main_dynamical} implies Theorem~\ref{thm_mrr}]
Assume that $A\subset\N$  has positive upper Banach density and let $\Phi$ be the \Folner{} sequence, $(X, \mu, T)$ the ergodic measure-preserving system, $E\subset X$ clopen, and $x\in X$ the point given by  Theorem~\ref{thm_correspondence_principle}.  
We inductively apply the limits to construct subsequences $(b_i')_{i\in\N}$ of $(b_i)_{i\in\N}$ and $(c_j')_{j\in\N}$ of $(c_j)_{j\in\N}$ such that  
$b_i'+c_j'\in A$ for all $i,j\in\N$.
First use the second iterated limit in~\eqref{eqn_fuzzy_Erdos_square} to choose
$b_1'\in (b_i)_{i\in\N}$ such that 
\begin{equation}
\label{eqn_choosing_first_b}
\lim_{j\to\infty}T^{b_1'+c_j}x\in E.
\end{equation}
Then use both~\eqref{eqn_choosing_first_b} and the first iterated limit in~\eqref{eqn_fuzzy_Erdos_square} to choose $c_1'\in (c_j)_{j\in\N}$ such that 
\[T^{b_1'+c_1'}x\in E \quad \text{ and } \quad \lim_{i\to\infty}T^{b_i+c_1'}x\in E.\]
We next choose $b_2'\in (b_i)_{i\in\N}$ sufficiently large such that 
\[T^{b_2'+c_1'}x\in E \quad \text{ and } \quad \lim_{j\to\infty}T^{b_2'+c_j}x\in E,\]
and then choose $c_2'\in (c_j)_{j\in\N}$ sufficiently large such that 
\[T^{b_1'+c_2'}x, T^{b_2'+c_2'}x\in E \quad \text{ and } \quad \lim_{i\to\infty}T^{b_i+c_2'}x\in E.\]
We then continue inductively to obtain sequences $(b_i')_{i\in\N}$ and $(c_j')_{j\in\N}$ with $T^{b_i'+c_j'}x\in E$ for all $i, j \in \N$.
Applying Theorem~\ref{thm_correspondence_principle} to the set $A$ it is immediate from property (iii) that $A$ contains the infinite sumset $\{ b_i'+c_j' : i,j \in \N \}$.
\end{proof}

\section{Proof of main dynamical result}

\begin{proof}[Proof of Theorem~\ref{thm_main_dynamical}]
After replacing the \Folner{} sequence $\Phi$ by a subsequence if necessary, there exists a point $y$ in the support of $\mu$ that is generic for $\mu$ along $\Phi$ (cf.~Remark~\ref{rem_generic_points_along_Folner}).
Further refining $\Phi$ if necessary, we can also assume that the pair $(x,y)$ is generic along $\Phi$ with respect to the product transformation $T\times T$ for a measure $\lambda$ on the product space $X\times X$.
Note that for any $f \in C(X)$, 
\[
\int_{X\times X} f(x_1) \d \lambda(x_1,x_2) =
\int_{X\times X} f(x_2) \d \lambda(x_1,x_2) = 
\int_X f \d \mu.
\]

Let $d\colon X\times X\to [0,\infty)$ denote a metric on $X$.
The central ingredient of the proof is the following claim.

\begin{claim}
\label{claim_1}
Let $Y\subset X\times X$ be a measurable set. If $\lambda\big((X\times E)\cap Y\big)>0$, then for every $\epsilon>0$ there exist infinitely many $b\in\N$ such that
\begin{equation}
\label{eqn_transposed_recurrence}
d(T^bx,y)< \epsilon\qquad\text{and}\qquad
\lambda\big((T^{-b}E\times X)\cap Y\big)>0. 
\end{equation}
\end{claim}

\begin{proof}[Proof of Claim]
\renewcommand{\qedsymbol}{$\triangle$}
Fix $\epsilon > 0$.
We write $f_c = \mathbb{E}(\1_E\mid \mathscr{K})$ for the orthogonal projection of $\1_E$ on the space $L^2(X,\mathscr{K},\mu)$ and $f_{\mathrm{wm}}=\1_E-f_c$, so that  $f_{\mathrm{wm}}$ is weak mixing by Theorem~\ref{thm_JdLG_decomposition}.
As noted in Remark~\ref{rem_support_of_conditional_expectation}, for $\mu$-almost every $x\in E$ we have $f_c(x)>0$, and hence
\[
\lambda\big((X\times E)\cap Y\big)>0~~\implies~~\big\langle (1\otimes f_c),\, \1_Y \big\rangle>0.
\]
Let $\eta=\langle (1\otimes f_c),\, \1_Y \rangle$. 
To prove~\eqref{eqn_transposed_recurrence}, it therefore suffices to find infinitely many $b\in\N$ such that
\[
\begin{aligned}
& \underbrace{d(T^bx,y)< \epsilon}_{[1]},\qquad  \underbrace{\big\|(T^b f_{\mathrm{c}}\otimes 1)\,-\, (1\otimes f_c)\big\|_{L^2(\lambda)}< \frac{\eta}{2}}_{[2]}, 
\\
&\qquad\qquad\text{and}\quad
\underbrace{\big|\big\langle (T^b f_{\mathrm{wm}}\otimes 1),\, \1_Y \big\rangle \big|< \frac{\eta}{2}}_{[3]}.
\end{aligned}
\]
Indeed, we may write
\[
\lambda\big((T^{-b}E\times X)\cap Y\big)
=
\langle T^b f_{\mathrm{c}} \otimes \1, \1_Y \rangle + \langle T^b f_{\mathrm{wm}} \otimes \1, \1_Y \rangle
\ge
\langle f_{\mathrm{c}} \otimes \1, \1_Y \rangle - \eta
\]
whenever $b$ satisfies [2] and [3].
Since $f_{\mathrm{c}}$ can be approximated by eigenfunctions, there exist $r\in\N$ and topological eigenfunctions $\chi_1,\ldots,\chi_r$ such that $f_{\mathrm{c}}^\prime=\chi_1+\ldots+\chi_r$ satisfies $\|f_{\mathrm{c}}^\prime-f_{\mathrm{c}}\|_{L^2}< \eta/6$.
The advantage of working with $f_{\mathrm{c}}^\prime$ over $f_{\mathrm{c}}$ is that $f_{\mathrm{c}}^\prime$ is continuous because each $\chi_i$ is continuous.
Let $\delta>0$ be sufficiently small to ensure that for $1\leq i\leq r$ we have $|\chi_i(y)-\chi_i(z)|<\eta/6r$ whenever $d(z,y)<\delta$.
Since $\chi_i$ is also an eigenfunction, we additionally obtain $\sup_{n\in\N}|\chi_i(T^ny)-\chi_i(T^nz)|<\eta/6r$ whenever $d(z,y)<\delta$. 
This implies the following equicontinuity property for the function $f_{\mathrm{c}}^\prime$: 
\begin{equation}
\label{eqn_equicontinuity}    
\sup_{n\in\N}|f_{\mathrm{c}}^\prime(T^ny)-f_{\mathrm{c}}^\prime(T^nz)|<\frac{\eta}{6}~~\text{whenever}~~d(z,y)<\delta.
\end{equation}

Next, consider the set $\mathcal{B}=\{b\in\N: d(T^b x,y)<\min\{\delta,\epsilon\}\}$.
We make three observations about $\mathcal{B}$. First, note that $[1]$ is satisfied for all $b\in\mathcal{B}$.  Second, it follows from~\eqref{eqn_equicontinuity} that 
\begin{equation}
\label{eqn_recurrence_along_S}   
\sup_{b\in\mathcal{B}}\sup_{n\in\N}|f_{\mathrm{c}}^\prime(T^{n+b}x)-f_{\mathrm{c}}^\prime(T^ny)|\leq\frac{\eta}{6}.
\end{equation}
Third, since $y$ belongs to the topological support of $\mu$ and $x$ is generic for $\mu$ along $\Phi$, it follows that the orbit of $x$ visits every neighborhood of $y$ with positive frequency along the \Folner{} sequence $\Phi$. In particular, we have
\begin{equation}
\label{eqn_density_of_S}   
\limsup_{M\to\infty}\frac{|\mathcal{B} \cap \Phi_M|}{|\Phi_M|} >0.
\end{equation}
Using $\|f_{\mathrm{c}}^\prime-f_{\mathrm{c}}\|_{L^2}< \eta/6$ and the fact that $(x,y)$ is generic for $\lambda$, and then~\eqref{eqn_recurrence_along_S} we get
\begin{align*}
\sup_{b\in\mathcal{B}}
\big\|(T^b f_{\mathrm{c}}\otimes 1)\,-\, & (1\otimes f_c)\big\|_{L^2}
\,<\,
\sup_{b\in\mathcal{B}}
\big\|(T^b f_{\mathrm{c}}^\prime\otimes 1)\,-\, (1\otimes f_c^\prime)\big\|_{L^2} + \frac{\eta}{3}
\\
&=
\sup_{b\in\mathcal{B}}\bigg(\lim_{N\to\infty}
\frac{1}{|\Phi_N|}\sum_{n\in \Phi_N}
\big|f_{\mathrm{c}}^\prime(T^{n+b}x) - f_c^\prime(T^ny)\big|^2 \bigg)^{\frac{1}{2}} + \frac{\eta}{3}
\,\leq\,
\frac{\eta}{2}.
\end{align*}
This proves that $[2]$ holds for all $b\in\mathcal{B}$.

Finally, since $f_{\mathrm{wm}}$ is a weak mixing function in $L^2(X,\mathscr{B}_{X},\mu)$, Remark~\ref{remark_wm} implies $(f_{\mathrm{wm}}\otimes 1)$ is a weak mixing function in $L^2(X\times X,\mathscr{B}_{X\times X},\lambda)$. In particular,
\[
\lim_{M\to\infty} \frac{1}{|\Phi_M|}\sum_{b\in\Phi_M} \big|\big\langle T^b f_{\mathrm{wm}}\otimes 1,\, \1_Y \big\rangle \big|=0.
\]
Combined with~\eqref{eqn_density_of_S}, this proves that there exist infinitely many $b\in \mathcal{B}$ such that $[3]$ holds. This completes the proof of the Claim.
\end{proof}

The proof of Theorem~\ref{thm_main_dynamical} now follows from a simple iteration of the Claim. 
Noting that $\lambda(X\times E)=\mu(E)>0$, we apply the claim with $Y=X\times E$ to find $b_1$ such that
\[
d(T^{b_1}x,y)<\frac{1}{2}\qquad\text{and}\qquad
\lambda\big(T^{-b_1}E\times E\big)>0.
\]
Applying the claim again, this time with $Y=T^{-b_1}E\times E$, we can find $b_2>b_1$ such that
\[
d(T^{b_2}x,y)<\frac{1}{4}\qquad\text{and}\qquad
\lambda\big((T^{-b_1}E\cap T^{-b_2}E)\times E\big)>0.
\]
Next, we apply the claim with $Y=(T^{-b_1}E\cap T^{-b_2}E)\times E$, and so on.
Continuing this procedure produces a sequence $b_1<b_2<\ldots\in\N$ such that $T^{b_i}x \to y$ as $i\to\infty$ and
\begin{equation}
\label{eqn_FIP_joining}
\lambda\big((T^{-b_1}E\cap\ldots\cap T^{-b_j}E)\times E\big)>0,\qquad \text{ for all } j\in \N.
\end{equation}
From the definition of $\lambda$, it follows that
\[
\lambda\big((T^{-b_1}E\cap\ldots\cap T^{-b_i}E)\times E\big)
=
\lim_{N\to\infty} \frac{1}{|\Phi_N|}\sum_{n\in\Phi_N}
\1_{E}(T^{n+b_1}x)\cdots \1_E(T^{n+b_i}x) \1_E(T^ny).
\]
In particular, for every $j\in\N$ we can find $c_j$ such that 
$T^{c_j}y\in E$ and $T^{c_j+b_i}x\in E$ for all $i\leq j$.
After replacing $(c_j)$ by a subsequence such that the limits $\lim_{j\to\infty} T^{c_j}x$ and $\lim_{j\to\infty} T^{c_j}y$ exist, and $(b_i)$ by a subsequence such that $\lim_{i\to\infty}\lim_{j\to\infty} T^{b_i+c_j}x$ exists, we have that~\eqref{eqn_fuzzy_Erdos_square} is satisfied.
\end{proof}

\bibliographystyle{references_style_file_for_arXiv}
\bibliography{references_for_arXiv}

\providecommand{\bysame}{\leavevmode\hbox to3em{\hrulefill}\thinspace}
\providecommand{\noopsort}[1]{}
\providecommand{\zbl}[1]{\href{http://www.zentralblatt-math.org/zmath/en/search/?q=an:#1}{Zbl~#1}}
\providecommand{\jfm}[1]{\href{http://www.emis.de/cgi-bin/JFM-item?#1}{JFM~#1}}
\providecommand{\arxiv}[1]{\href{http://www.arxiv.org/abs/#1}{arXiv~#1}}
\providecommand{\doi}[1]{\url{https://doi.org/#1}}
\providecommand{\href}[2]{#2}
\begin{thebibliography}{KMRR24b}

\bibitem[EW11]{EW11}
\bgroup\scshape{}M.~Einsiedler\egroup{} and \bgroup\scshape{}T.~Ward\egroup{},
  \emph{Ergodic theory with a view towards number theory}, \emph{Graduate Texts
  in Mathematics} \textbf{259}, Springer London, London, 2011.
  \doi{10.1007/978-0-85729-021-2}.

\bibitem[Fur77]{Furstenberg77}
\bgroup\scshape{}H.~Furstenberg\egroup{}, Ergodic behavior of diagonal measures
  and a theorem of {Szemer{\'e}di} on arithmetic progressions,  \emph{Journal
  d'Analyse Math{\'e}matique} \textbf{31} (1977), 204--256.
  \doi{10.1007/BF02813304}.

\bibitem[Fur81]{Furstenberg81a}
\bgroup\scshape{}H.~Furstenberg\egroup{}, \emph{Recurrence in ergodic theory
  and combinatorial number theory}, \emph{M. B. Porter Lectures}, Princeton
  University Press, Princeton, NJ, 1981.

\bibitem[Hos19]{Host19}
\bgroup\scshape{}B.~Host\egroup{}, A short proof of a conjecture of
  {Erd{\H{o}}s} proved by {Moreira}, {Richter} and {Robertson},  \emph{Discrete
  Analysis} (2019), Paper No. 19, 10 pp. \doi{10.19086/da.11129}.

\bibitem[HK18]{HK18}
\bgroup\scshape{}B.~Host\egroup{} and \bgroup\scshape{}B.~Kra\egroup{},
  \emph{Nilpotent structures in ergodic theory}, \emph{Mathematical Surveys and
  Monographs} \textbf{236}, American Mathematical Society, Providence, RI,
  2018. \doi{10.1090/surv/236}.

\bibitem[KvN32]{Koopman_vonNeumann32}
\bgroup\scshape{}B.~O. Koopman\egroup{} and \bgroup\scshape{}J.~von
  Neumann\egroup{}, Dynamical systems of continuous spectra,  \emph{Proceedings
  of the National Academy of Sciences of the United States of America}
  \textbf{18} no.~3 (1932), 255--263. \doi{10.1073/pnas.18.3.255}.

\bibitem[KMRR24a]{KMRR24a}
\bgroup\scshape{}B.~Kra\egroup{}, \bgroup\scshape{}J.~Moreira\egroup{},
  \bgroup\scshape{}F.~K. Richter\egroup{}, and
  \bgroup\scshape{}D.~Robertson\egroup{}, Infinite sumsets in sets with
  positive density,  \emph{Journal of the American Mathematical Society}
  \textbf{37} no.~3 (2024), 637--682. \doi{10.1090/jams/1030}.

\bibitem[KMRR24b]{KMRR24b}
\bgroup\scshape{}B.~Kra\egroup{}, \bgroup\scshape{}J.~Moreira\egroup{},
  \bgroup\scshape{}F.~K. Richter\egroup{}, and
  \bgroup\scshape{}D.~Robertson\egroup{}, A proof of {Erd{\H{o}}s's} {$B+B+t$}
  conjecture,  \emph{Communications of the American Mathematical Society}
  \textbf{4} no.~10 (2024), 480--494. \doi{10.1090/cams/34}.

\bibitem[KMRR25]{KMRR25}
\bgroup\scshape{}B.~Kra\egroup{}, \bgroup\scshape{}J.~Moreira\egroup{},
  \bgroup\scshape{}F.~K. Richter\egroup{}, and
  \bgroup\scshape{}D.~Robertson\egroup{}, Problems on infinite sumset
  configurations in the integers and beyond,  \emph{Bulletin of the American
  Mathematical Society} \textbf{62} no.~4 (2025), 537--574.
  \doi{10.1090/bull/1868}.

\bibitem[KMRR26]{KMRR26}
\bgroup\scshape{}B.~Kra\egroup{}, \bgroup\scshape{}J.~Moreira\egroup{},
  \bgroup\scshape{}F.~K. Richter\egroup{}, and
  \bgroup\scshape{}D.~Robertson\egroup{}, The density finite sums theorem,
  \emph{Inventiones Mathematicae} \textbf{243} no.~1 (2026), 1--31.
  \doi{10.1007/s00222-025-01371-8}.

\bibitem[Kre85]{Krengel85}
\bgroup\scshape{}U.~Krengel\egroup{}, \emph{Ergodic theorems}, \emph{De Gruyter
  Studies in Mathematics} \textbf{6}, Walter de Gruyter \& Co., Berlin, 1985,
  With a supplement by Antoine Brunel. \doi{10.1515/9783110844641}.

\bibitem[MRR19]{MRR19}
\bgroup\scshape{}J.~Moreira\egroup{}, \bgroup\scshape{}F.~K. Richter\egroup{},
  and \bgroup\scshape{}D.~Robertson\egroup{}, A proof of a sumset conjecture of
  {Erd{\H{o}}s},  \emph{Annals of Mathematics} \textbf{189} no.~2 (2019),
  605--652. \doi{10.4007/annals.2019.189.2.4}.

\end{thebibliography}

\end{document}